\documentclass[12pt,a4paper]{amsart}
\usepackage{amsmath, amssymb, amsfonts,enumerate}
\usepackage[all]{xy}
\usepackage{amscd}

\newtheorem{theorem}{Theorem}[section]
\newtheorem{lemma}[theorem]{Lemma}
\newtheorem{corollary}[theorem]{Corollary}
\newtheorem{proposition}[theorem]{Proposition}

\theoremstyle{definition}
\newtheorem{definition}[theorem]{Definition}
\newtheorem*{definition*}{Definition}

\theoremstyle{remark}
\newtheorem*{remark}{Remark}

\numberwithin{equation}{section}

\newcommand {\Z}{\mathbb{Z}} 

\newcommand{\FF}{\mathcal{F}}
\newcommand{\PP}{\mathcal{P}}

\newcommand{\GG}{\mathcal{G}}

\DeclareMathOperator{\ent}{ent}

\begin{document}
\title[The Myhill property for strongly irreducible subshifts]{The Myhill property for strongly irreducible subshifts over amenable groups}

\author{Tullio Ceccherini-Silberstein}
\address{Dipartimento di Ingegneria, Universit\`a del Sannio, C.so
Garibaldi 107, 82100 Benevento, Italy}
\email{tceccher@mat.uniroma1.it}
\author{Michel Coornaert}
\address{Institut de Recherche Math\'ematique Avanc\'ee, UMR 7501,                           
Universit\'e  de Strasbourg et CNRS, 7 rue Ren\'e-Descartes,
67000 Strasbourg, France}
\email{coornaert@math.unistra.fr}
\date{September 3rd, 2010}
\begin{abstract}
Let $G$ be an amenable group and let $A$ be a finite set.
We prove that if $X \subset A^G$ is a strongly irreducible subshift 
then $X$ has the Myhill property, that is, every pre-injective cellular automaton $\tau \colon X \to X$ is surjective.
      \end{abstract}

\maketitle

\section{Introduction}
\label{s:introduction}

Let $G$ be a group and let $A$ be a finite set.
We equip  the set $A^G = \prod_{g \in G} A =  \{x \colon G \to A\}$  with its \emph{prodiscrete} topology, that is, with the product topology obtained by taking the discrete topology on each factor $A$ of $A^G$.
The elements of $A^G$ are called the \emph{configurations} over the group $G$ and the 
\emph{alphabet} $A$.
The $G$-\emph{shift} on $A^G$ is the continuous left action of $G$ on $A^G$ defined by
$gx(h) = x(g^{-1}h)$ for all $g,h \in G$ and $x \in A^G$.
 A closed $G$-invariant subset of $A^G$ is called a \emph{subshift}.
 The set $A^G$ is a subshift of itself which is traditionally referred to  as the \emph{full shift}.
 \par
  A subshift $X \subset A^G$ is called \emph{strongly irreducible} if there is a finite subset 
 $\Delta \subset G$ satisfying the following property:
 if $\Omega_1$ and $\Omega_2$   are finite subsets of $G$ such that there exists no element $g \in \Omega_2$ such that $g \Delta$ meets  $\Omega_1$,  then, given any two configurations $x_1,x_2 \in X$, there exists  a configuration $x \in X$ which coincides with $x_1$ on $\Omega_1$ and with $x_2$ on $\Omega_2$.
   A subshift $X \subset A^G$ is said to be \emph{of finite type} if there exists a finite subset $D \subset G$ and a subset $P \subset A^D$ such that $X$ consists of all the configurations 
   $x \in A^G$  such that   the restriction of $g^{-1}x$ to $D$ belongs to $P$ for all $g \in G$.
   Such a subset $D$ is then called a \emph{defining window} for $X$.  
   \par
A
 \emph{cellular automaton} on a subshift $X \subset A^G$ is a map $\tau \colon X \to X$ which is continuous (for the prodiscrete topology) and commutes with the $G$-shift (i.e., such that 
$\tau(gx) = g\tau(x)$ for all $g \in G$ and $x \in X$).
A cellular automaton $\tau \colon X \to X$ on a subshift $X \subset A^G$ is called \emph{pre-injective}
if the equality $\tau(x_1) = \tau(x_2)$ implies $x_1 = x_2$ whenever the configurations $x_1,x_2 \in X$ coincide outside of a finite subset of $G$.
Every injective cellular automaton is pre-injective but 
there are pre-injective cellular automata which are not injective.
For example, by taking $G = \Z$ and $A = \Z/2\Z$,
the map $\tau \colon A^G \to A^G$ defined by $\tau(x)(n) = x(n) + x(n + 1)$ is a cellular automaton which is pre-injective but not injective.
\par
The \emph{Garden of Eden theorem}
says that if $\tau \colon A^G \to A^G$ is a cellular automaton on the full shift $A^G$, where $G$ is an amenable group and $A$ is a finite set, then $\tau$ is surjective if and only if it is pre-injective
(see Subsection \ref{ss:amenable-groups} for the definition of amenable groups).
It was first established in the special case $G = \Z^2$ by  E.F. Moore 
 \cite{moore} who proved the implication ``surjective $  \Rightarrow$ pre-injective" and by  J. Myhill 
 \cite{myhill} who proved the converse implication.
  The Garden of Eden theorem was subsequently extended to all finitely generated amenable groups in \cite{ceccherini} (see \cite{induction} for general amenable groups).   
\par
One says that a subshift $X \subset A^G$ has the \emph{Moore property} if every surjective cellular automaton $\tau \colon X \to X$ is pre-injective and that it has the \emph{Myhill property} if every pre-injective cellular automaton $\tau \colon X \to X$ is surjective.
 In \cite{fiorenzi-strongly} F. Fiorenzi proved a Garden of Eden theorem for strongly irreducible subshifts of finite type $X \subset A^G$, where $G$ is a finitely generated  amenable group and $A$ is a finite set.
In other words, such subshifts have both the Moore and the Myhill property.
\par
The \emph{even subshift} is the subshift $X \subset \{0,1\}^\Z$
formed by  all bi-infinite sequences of $0$s and $1$s in which every chain of $0$s which is bounded by two $1$s has even length.
In \cite[Section 3]{fiorenzi-sofic}, Fiorenzi gave an example of a cellular automaton over the even subshift 
 which is surjective but not pre-injective.
  As the even subshift is strongly irreducible and $\Z$ is amenable, this shows that a strongly irreducible subshift 
 $X \subset A^G$, with $G$ amenable and $A$ finite,
 may fail to have the Moore property.
However, it turns out that  strongly irreducible  subshifts over amenable groups and finite alphabets have always the Myhill property (even if its is not of finite type).
This is the main result of the present paper:

\begin{theorem}
\label{t:main-theorem}
Let $G$ be a (possibly uncountable) amenable group and let $A$ be a finite set.
Then every strongly irreducible subshift $X \subset A^G$
  has the Myhill property.
  \end{theorem}

The proof of Theorem \ref{t:main-theorem} relies on the entropic properties of strongly irreducible subshifts over amenable groups. 
More precisely, the first step consists in showing that  if $\tau \colon X \to X$ is a pre-injective cellular automaton over a strongly irreducible subshift $X \subset A^G$, with $G$ amenable and $A$ finite,
then the entropy of the subshift $\tau(X)$ is equal to that of $X$ (see Theorem \ref{t:pre-inj-si-source}). We then conclude that 
$\tau(X) = X$ by using the fact that if $Y$ is any proper subshift of $X$ then the entropy of $Y$ is  strictly smaller than the entropy  of $X$ (Proposition \ref{p:entropy-increasing}).
\par
When $G = \Z$ whe have the following characterization of strongly irreducible subshifts
(see Section \ref{sec:strongly-over-Z} for the definition of the language $L(X)$ associated with a subshift $X$):

\begin{proposition}
\label{p;equivalences-strong-irreducibility}
Let $A$ be a finite set and let $X \subset A^\Z$ be a subshift. 
Then the following conditions are equivalent:
\begin{enumerate}[{\rm (a)}]
\item 
$X$ is strongly irreducible;
\item 
there is an integer $N_0 \geq 0$ such that, for all $u,v \in L(X)$ and for every $N \geq N_0$, 
there exists a word $w \in A^*$ of length $N$ satisfying $uwv \in L(X)$.
\end{enumerate}
\end{proposition}

As an application, we have the following (see Section \ref{sec:strongly-over-Z} for the definition of a sofic subshift):

\begin{corollary} 
\label{c:sofic-strongly-irr}
Let $A$ be a finite set and let $X \subset A^\Z$ be a sofic subshift. 
Then $X$ is strongly irreducible if and only if it is topologically mixing.
\end{corollary}

From Theorem \ref{t:main-theorem} and Corollary \ref{c:sofic-strongly-irr}, it follows that every topologically mixing sofic subshift over $\Z$ has the Myhill property. 
In fact, Fiorenzi \cite[Corollary 2.21]{fiorenzi-sofic} proved the stronger result that every irreducible sofic subshift over $\Z$ has the Myhill property.
As every subshift of finite type over $\Z$ is sofic, this implies in particular that every irreducible subshift of finite type over $\Z$ has the Myhill property.
A trivial example of a  subshift of finite type over $\Z$ which does not have the Myhill property is provided by the subshift $X = \{x_0,x_1\} \subset \{0,1\}^\Z$, where $x_0$ and $x_1$ are the two constant configurations defined by $x_0(n) = 0 $ and $x_1(n) = 1 $ for all $n \in \Z$. Indeed, the cellular automaton $\tau \colon X \to X$ given by $\tau(x_0) = \tau(x_1) = x_0$ is pre-injective but not surjective.
On the other hand, there exist topologically mixing subshifts over $\Z$ which are not strongly irreducible. An example of such a subshift is provided by the subshift $X \subset \{0,1\}^\Z$ consisting of all bi-infinite sequences of $0$s and $1$s in which there is no word of the form $01^h0^k1$, where $h$ and $k$ are positive integers with  $h \geq k$, but
we do not know whether this subshift has the Myhill property or not. 
Finally, let us remark that there exist topologically mixing subshifts of finite type over the group $\Z^2$ which are not strongly irreducible.
Indeed,   B. Weiss gave in \cite[Section 4]{weiss-sgds} an example of a topologically mixing subshift of finite type $X \subset A^{\Z^2}$, with $A$ of cardinality $4$,
admitting an injective (and therefore pre-injective) cellular automaton $\tau \colon X \to X$ which is not surjective. Such a subshift is not strongly irreducible by 
Theorem \ref{t:main-theorem} since it does not satisfy the Myhill property.
\par 
The paper is organized as follows.
Section \ref{s:background} gathers preliminary  material.
In Section \ref{s:strongly} we establish general properties of strongly irreducible subshifts. We prove in particular that every strongly irreducible subshift is topologically mixing and that strong irreducibility is a conjugacy invariant for subshifts. Section \ref{s:entropic-properties} is devoted to the study of entropic properties of strongly irreducible subshifts.
This section contains several results which may be of independent interest.
It is shown in particular that any non-trivial strongly irreducible subshift $X \subset A^G$, with $G$ amenable and $A$ finite, has positive entropy (Proposition \ref{p:positive-entropy}). Section \ref{s:proof-main-result} contains the proof of Theorem  \ref{t:main-theorem}.
In the final section, we present the proofs of Proposition \ref{p;equivalences-strong-irreducibility} and of Corollary \ref{c:sofic-strongly-irr}. 
  
\section{Background material}
\label{s:background}
In this section we introduce the notation and collect definitions and basic facts that will be used in the sequel. Some proofs of well-known results are given for the convenience of the reader.

\subsection{General notation}
We use $\vert \cdot \vert$ to denote cardinality of finite sets.
\par
Let $G$ be a group and let $A$ be a finite set.
For $\Omega \subset G$, we denote by $\pi_\Omega \colon A^G \to A^\Omega$ the projection map.
For $x \in A^G$ , we denote by $x\vert_\Omega$ the restriction of $x$ 
to $\Omega$, that is, the element $x\vert_\Omega = \pi_\Omega(x) \in A^\Omega$ given by $x\vert_\Omega(g) = x(g)$ for all $g \in \Omega$. 
For $X \subset A^G$, we define $X_\Omega \subset A^\Omega$ by
$$
X_\Omega = \pi_\Omega(X) = \{x\vert_\Omega : x \in X\}
$$ 
Note that it follows from the definition of the prodiscrete topology on $A^G$ that a subset $X \subset A^G$ is closed in $A^G$ if and only if it satisfies the following condition:
if an element $x \in A^G$ satisfies $x\vert_\Omega \in X_\Omega$ for every finite subset $\Omega \subset G$, then one has $x \in X$.

\subsection{Neighborhoods}
Let $G$ be a group.
Let $\Delta$ and $\Omega$ be  subsets of $G$. 
 The $\Delta$-\emph{neighborhood} of $\Omega$ in $G$  is   the subset
 $\Omega^{+ \Delta} \subset G$ defined by   
   \begin{equation*}
  \Omega^{+\Delta} = \{g \in G : g\Delta \cap \Omega \not= \varnothing \} = \Omega\Delta^{-1}.
 \end{equation*}

Note that  $\Omega \subset \Omega^{+\Delta}$ if $1_G \in \Delta$. Note also that $\Omega^{+\Delta_1\Delta_2}=(\Omega^{+\Delta_2})^{+\Delta_1}$ for all $\Omega, \Delta_1, \Delta_2 \subset G$.
On the other hand, we have $\Omega^{+\Delta} \subset \Omega'^{+ \Delta'}$ whenever 
$\Omega \subset \Omega' \subset G$ and  $\Delta \subset \Delta' \subset G$.
Finally, observe that $\Omega^{+ \Delta}$ is finite if both $\Omega$ and $\Delta$ are finite. 

\subsection{Cellular automata}
Let $G$ be a group and let $A$ and $B$ be two finite sets.
A map $\tau \colon X \to Y$ between subshifts $X \subset A^G$ and $Y \subset B^G$ is called a \emph{cellular automaton} if 
$\tau$ is continuous (with respect to the prodiscrete topologies on $A^G$ and $B^G$) and 
$G$-equivariant (i.e., such that $\tau(gx) = g\tau(x)$ for all $g \in G$ and $x \in X$).
It follows from the Curtis-Hedlund theorem \cite{hedlund-endomorphisms} that a map $\tau \colon X \to Y$ is a cellular automaton if and only if there exist a finite set $M \subset G$ and a map 
$\mu \colon A^G \to B$ such that 
\begin{equation}
\label{e:tau-x-g} 
  \tau(x)(g) = \mu \circ \pi_M (g^{-1}x) \quad \text{for all } x \in X \text{ and } g \in G.
\end{equation}
  Such a set $M$ is called a \emph{memory set}
    and $\mu$ is called a \emph{local defining map}
   for the cellular automaton $\tau$.
Note that if $M$ is a memory set for the cellular automaton $\tau$ then any finite subset of $G$ containing $M$ is also a memory set for $\tau$. 
  \par
It immediately follows from the preceding  characterization of cellular automata that a map
   $\tau \colon X \to Y$ between subshifts $X \subset A^G$ and $Y \subset B^G$ is a cellular automaton if and only if there exists a cellular automaton $\sigma \colon A^G \to B^G$ whose restriction to $X$ coincides with $\tau$. 
\par
Suppose that $\tau \colon X \to Y$ is a cellular automaton between the subshifts
$X \subset A^G$ and $Y \subset B^G$. Then its image $\tau(X)$ is a subshift of $B^G$.
Indeed, $\tau(X)$ is closed in $B^G$ by the compactness of $X$ and the continuity of $\tau$, and it is $G$-invariant by the $G$-equivariance of $\tau$ and the $G$-invariance of $X$.
Note also that if the cellular automaton $\tau \colon X \to Y$ is bijective
then its inverse map $\tau^{-1} \colon Y \to X$ is itself a cellular automaton since $\tau^{-1}$ is $G$-equivariant by the $G$-equivariance of $\tau$ and continuous by the continuity of $\tau$ and the compactness of $X$.
\par
Two subshifts $X \subset A^G$ and $Y \subset B^G$ are called \emph{conjugate}
if there exists a $G$-equivariant homeomorphism from $X$ onto $Y$, i.e.,
if there exists a bijective cellular automaton $\tau \colon X \to Y$.

\par
We will frequently use the following fact, which is an immediate consequence of \eqref{e:tau-x-g}:
if $M$ is a memory set for the cellular automaton $\tau \colon X \to Y$, then, given  $x \in X$  and $g \in G$, the element  $\tau(x)(g) \in B$ depends only on the restriction of $x$ to $gM$.
This implies in particular that if $x_1,x_2 \in X$ are two configurations which coincide outside of a subset $\Omega \subset G$ (resp. on $\Omega^{+M^{-1}}$)  then the configurations $\tau(x_1)$ and $\tau(x_2)$ coincide outside of $\Omega^{+M}$
(resp. on $\Omega$).
\par
 A cellular automaton $\tau \colon X \to Y$ between subshifts $X \subset A^G$ and 
 $Y \subset B^G$ is called \emph{pre-injective}
if the equality $\tau(x_1) = \tau(x_2)$ implies $x_1 = x_2$ whenever the configurations 
$x_1,x_2 \in X$ coincide outside of a finite subset of $G$.
The term \emph{pre-injective} was introduced by M. Gromov in \cite[Section 8]{gromov-esav}. 

\subsection{Amenable groups}
\label{ss:amenable-groups}
There are many equivalent definitions of amenability for groups in the literature (see for example \cite{greenleaf}, \cite{paterson}).
Here we shall use the following one, which is known as the \emph{F\o lner condition}:

\begin{definition}
A group $G$ is called \emph{amenable} if there exist a directed set $J$ and a family 
$\FF = (F_j)_{j \in J}$ of nonempty finite subsets of $G$ indexed by $J$ satisfying
\begin{equation}
\label{e:folner-net}
\lim_j \frac{\vert F_j^{+E} \setminus F_j \vert}{\vert F_j \vert} = 0 
\quad \text{ for any finite subset } E \subset G.
\end{equation}
  Such a family $\FF$ is called a \emph{F\o lner net} for $G$.
\end{definition}

All locally finite groups, all solvable groups (and therefore all abelian groups), and all finitely generated groups of subexponential growth are amenable.
The free group of rank $2$ provides an example of a non-amenable group.
As the class of amenable groups is closed under taking subgroups, it follows that  if a group $G$ contains a nonabelian free subgroup then $G$ is not amenable.

\subsection{Entropy}
Let $G$ be an amenable group and let $A$ be a finite set.
Let $\FF = (F_j)_{j \in J}$ be a F\o lner net for $G$.
   The \emph{entropy} $\ent_\FF(X)$  of a subset $X \subset A^G$   is the quantity
\begin{equation}
\label{e:entropy}
\ent_\FF(X) = \limsup_j \frac{\log \vert  X_{F_j} \vert}{\vert F_j \vert}.
\end{equation}
Note that one always has $\ent_\FF(X) \leq \log \vert A \vert = \ent_\FF(A^G)$ and $\ent_\FF(X) \leq \ent_\FF(Y)$ whenever $X \subset Y \subset A^G$. 

\begin{remark}
It can be shown that, when $X$ is a $G$-invariant subset of $A^G$, the $\limsup$ in
\eqref{e:entropy} is in fact a true limit and that it does not depend on $\FF$, 
but we do not need these two facts in the sequel.
\end{remark}

An important (and well known)  property of cellular automata is that they cannot increase entropy:

\begin{proposition}
\label{p:ca-decrease-entropy}
Let $G$ be an amenable group and let $\FF = (F_j)_{j \in J}$ be a F\o lner net for $G$.   
Let $A$ and $B$ be two finite sets.
Let $\tau \colon X \to Y$ be a cellular automaton between subshifts
  $X \subset A^G$ and $Y \subset B^G$.
   Then one has
$\ent_\FF(\tau(Z)) \leq \ent_\FF(Z)$ for every subset $Z \subset X$.
\end{proposition}

\begin{proof}
Let $W = \tau(Z)$.
Choose a memory set $M$ for $\tau$ with $1_G \in M$.
 Then, observe that $\tau$ induces, by restriction, a surjective map $\tau_j \colon Z_{F_j^{+M^{-1}}} \to 
W_{F_j}$. This implies
\begin{equation}
\label{e:bound-for-W}
\vert W_{F_j} \vert \leq \vert Z_{F_j^{+M^{-1}}} \vert  \quad \text{for all } j \in J.
\end{equation}
 Now, as $Z_{F_j^{+M^{-1}}} \subset Z_{F_j} \times B^{F_j^{+M^{-1}} \setminus F_j}$, we get
$$
\log \vert Z_{F_j^{+M^{-1}}} \vert \leq \log \vert Z_{F_j} \vert + \vert F_j^{+M^{-1}} \setminus F_j \vert \cdot \log \vert B \vert.
$$  
 Using \eqref{e:bound-for-W},   this gives us
$$
 \log \vert W_{F_j} \vert  \leq
 \log \vert Z_{F_j} \vert +   \vert F_j^{+ M^{-1}} \setminus F_j \vert \cdot \log \vert B \vert. 
$$
After dividing by $\vert F_j \vert$ and taking the limsup,
we finally get $\ent_\FF(W) \leq \ent_\FF(Z)$ since 
$\vert F_j^{+ M^{-1}} \setminus F_j \vert/\vert F_j \vert$ tends to $0$ by \eqref{e:folner-net}.
\end{proof}
 \begin{corollary}
\label{c:entropy-invariant}
Let $G$ be an amenable group and let $\FF = (F_j)_{j \in J}$ be a F\o lner net for $G$.   
Let $A$ and $B$ be two finite sets.
Suppose that $X \subset A^G$ and $Y \subset B^G$ are two subshifts such that
there exists a bijective cellular automaton $\tau \colon X \to Y$. 
Then one has $\ent_\FF(X) = \ent_\FF(Y)$.  
\end{corollary}

\begin{proof}
We have $\ent_\FF(Y) \leq \ent_\FF(X)$ by Proposition \ref{p:ca-decrease-entropy}.
On the other hand, as observed above,  the inverse map $\tau^{-1} \colon Y \to X$ is also a cellular automaton. 
Therefore, we obtain   $\ent_\FF(X) \leq \ent_\FF(Y)$ by applying again Proposition 
\ref{p:ca-decrease-entropy}.  
\end{proof}

\subsection{Tilings} (see \cite[Section 2]{garden})
\label{ss;tilings}
Let $G$ be a group. Given two subsets $E$ and $E'$  of $G$, one says that a subset $T \subset G$ is  an $(E,E')$-\emph{tiling} if 
 the subsets $gE$, $g \in T$, are pairwise disjoint
 and $G = \bigcup_{g \in T} gE'$.
 \par
The following statement may be deduced from Zorn's lemma (cf. \cite[Lemma 2.2]{garden}):

\begin{lemma} 
\label{l;tilings-exist}
Let $G$ be a group. Let $E$ be a nonempty subset of $G$ and let
$E' = EE^{-1} = \{ab^{-1} : a,b \in E\}$. Then $G$ contains an $(E,E')$-tiling.
\qed 
\end{lemma}

We shall use the following lower estimate for the asymptotic growth of tilings with respect to F\o lner nets in amenable groups (see  \cite[Lemma 4.3]{garden} for the proof):

\begin{lemma} 
\label{l;tiling}
Let $G$ be an amenable group and let $(F_j)_{j \in J}$ be a right F\o lner net for $G$.
Let $E$ and $E'$ be finite subsets of $G$ and suppose that
$T \subset G$ is an $(E,E')$-tiling. 
For each $j \in J$, let  
$T_j $ be the subset of $T$ defined by $T_j = \{g \in T: gE \subset F_j\}$.
Then there exist a real number $\alpha> 0$ and an element $j_0 \in J$ such that 
$\vert T_j \vert \geq \alpha \vert F_j \vert$
for all $j \geq j_0$. 
\qed
\end{lemma}

\section{General properties of strongly irreducible subshifts}
\label{s:strongly}

 Let $G$ be a group and let $A$ be a finite set. 

\begin{definition}
\label{def:delta-irred-subshift}
Let $\Delta$ be a finite subset of $G$.
A subshift $X \subset A^G$ is said to be $\Delta$-\emph{irreducible} if it satisfies the following condition:
if $\Omega_1$ and $\Omega_2$ are finite subsets of $G$ such that  
\begin{equation}
\label{e:delta-neighbor-disjoint}
\Omega_1^{+\Delta} \cap \Omega_2 = \varnothing,
\end{equation}
 then, given any two configurations $x_1$ and $x_2$ in $X$, there exists a configuration $x \in X$ which satisfies
$x\vert_{\Omega_1} = x_1\vert_{\Omega_1}$ and $x\vert_{\Omega_2} = x_2\vert_{\Omega_2}$. 
\end{definition}

Note that if a subshift $X \subset A^G$ is $\Delta$-irreducible for some finite subset $\Delta \subset G$, then
$X$ is $\Delta'$-irreducible for any finite subset $\Delta' \subset G$ such that $\Delta \subset \Delta'$.

\begin{definition}
\label{def:strongly-irred-subshift}
A subshift $X \subset A^G$ is called \emph{strongly irreducible} if there exists a finite subset 
$\Delta \subset G$  such that $X$ is $\Delta$-irreducible.
\end{definition}

\begin{remark}
In the case when the group $G$ is finitely generated,
the above definition is equivalent to the one given in \cite[Definition 4.1]{fiorenzi-strongly}
(this immediately follows from the fact that if $G$ is endowed with the word metric associated with a finite symmetric generating subset $S \subset G$, then every ball of $G$ is finite and any finite subset of $G$ is contained in some ball).  
\end{remark}

Recall the following classical definitions from  topological dynamics. Suppose that a group $G$ acts continuously on a topological space $X$. One says that the action of $G$ on $X$ is
\emph{topologically transitive} if, for any pair of nonempty open subsets $U$ and $V$ of $X$, there exists an element $g \in G$ such that $U \cap gV \neq \varnothing$.
One says that the action of $G$ on $X$ is \emph{topologically mixing} if, for any pair of nonempty open subsets $U$ and $V$ of $X$, there exists a finite subset $F \subset G$ such that $U  \cap gV \neq \varnothing$ for all $g \in G \setminus F$.
\par
One says that a subshift $X \subset A^G$ is   \emph{irreducible} if the action of $G$ on $X$ is topologically transitive.
This is equivalent to the fact that $X$ satisfies the following condition:
for any finite subset $\Omega \subset G$ and any two configurations  $x_1, x_2 \in X $,
there exist an element $g \in G$ and   a configuration $x \in X$ such that $x\vert_\Omega = x_1\vert_{\Omega}$ and $ x\vert_{g\Omega} = x_2\vert_{g\Omega}$. 
 \par
One says that a subshift $X \subset A^G$ is \emph{topologically mixing}  if the action of $G$ on $X$ is topologically mixing.
This is equivalent to the fact that $X$ satisfies the following condition:
for any finite subset $\Omega \subset G$ and any two configurations  $x_1, x_2 \in X $,
there exists a finite subset $F \subset G$ such that, for all $g \in G \setminus F$, there exists a configuration $x \in X$ such that $x\vert_\Omega = x_1\vert_{\Omega}$ and $ x\vert_{g\Omega} = x_2\vert_{g\Omega}$. 
Note that if $G$ is finite then every subshift $X \subset A^G$ is topologically mixing and that if $G$ is infinite then every topologically mixing subshift $X \subset A^G$ is irreducible.

\begin{proposition}
\label{p:strongly-irred-implies-top-mixing}
Let $G$ be a group and let $A$ be a finite set.
Then every strongly irreducible subshift $X \subset A^G$ is topologically mixing.
\end{proposition}

\begin{proof}
Let $X \subset A^G$ be a strongly irreducible subshift.
Thus, there is a finite subset $\Delta \subset G$ such that $X$ is $\Delta$-irreducible.
Suppose  that we are given a finite subset $\Omega \subset G$ and two configurations $x_1,x_2 \in X $.
Consider the finite subset $F \subset G$ defined by $F = \Omega\Delta^{-1}\Omega^{-1}$.  
If  $g \in G \setminus F$, then we have
$$
\Omega^{+\Delta} \cap (g \Omega)  = \varnothing.
$$
 Since $X$ is $\Delta$-irreducible, this implies that there exists a configuration $x \in X$ such that 
 $x\vert_\Omega = x_1\vert_\Omega$ and $x\vert_{g \Omega} =  x_2\vert_{g\Omega}$.
 This shows that $X$ is topologically mixing. 
 \end{proof}

\begin{proposition}
\label{p:tau-strongly-is-strongly}
Let $G$ be a group and let $A$  be a finite set.
Let $X \subset A^G$ be a subshift.
Suppose that there exist a finite set $B$, a strongly irreducible subshift $Y \subset B^G$, and a surjective cellular automaton $\tau \colon Y \to X$.
Then the subshift $X \subset A^G$ is strongly irreducible.
\end{proposition}

\begin{proof}
Let $\Delta$ be a finite subset of $G$ such that $Y$ is $\Delta$-irreducible
and let $M \subset G$ be a memory set for $\tau$.
We claim that $X$ is $(M\Delta M^{-1})$-irreducible. 
Indeed, let $\Omega_1$ and $\Omega_2$ be two
finite subsets of $G$ such that $\Omega_1^{+(M\Delta M^{-1})} \cap \Omega_2 = \varnothing$
and let $x_1, x_2 \in X$.
Since $\Omega_1^{+(M\Delta M^{-1})}= \Omega_1 M \Delta^{-1} M^{-1}$ and
$\Omega_1 M \Delta^{-1} = (\Omega_1^{+M^{-1}})^{+\Delta}$, we deduce that
\begin{equation*}
(\Omega_1^{+M^{-1}})^{+\Delta} \cap \Omega_2^{+M^{-1}} = \varnothing.
\end{equation*}
Since $\tau$ is surjective, we can find two configurations $y_1$ and $y_2$ in $Y$
such that $\tau(y_1) = x_1$ and $\tau(y_2) = x_2$. As $Y$ is $\Delta$-irreducible,
there exists a configuration $y \in Y$ such that 
\begin{equation}
\label{e;restriction-de-y-a-1}
y\vert_{\Omega_1^{+M^{-1}}} = y_1\vert_{\Omega_1^{+M^{-1}}}
\quad \text{ and} \quad
 y\vert_{\Omega_2^{+M^{-1}}} = y_2\vert_{\Omega_2^{+M^{-1}}}.
\end{equation} 
As the values of $\tau(y)$ on a subset $\Omega \subset G$ only
depend on the values of $y$ on $\Omega^{+M^{-1}}$,  we deduce 
from \eqref{e;restriction-de-y-a-1} that the configuration $x = \tau(y) \in A^G$ satisfies
$$
x\vert_{\Omega_1} = \tau(y)\vert_{\Omega_1} = \tau(y_1)\vert_{\Omega_1} = x_1\vert_{\Omega_1}
$$
and  
$$
x\vert_{\Omega_2} = \tau(y)\vert_{\Omega_2} = \tau(y_2)\vert_{\Omega_2} = x_2\vert_{\Omega_2}.
$$
This proves our claim. Thus $X$ is strongly irreducible.
 \end{proof}

As the full shift $A^G$, viewed as a subshift of itself,  is $\{1_G\}$-irreducible and therefore strongly irreducible for any group $G$ and any finite set $A$,   we immediately deduce from  
Proposition \ref{p:tau-strongly-is-strongly} the following result:

\begin{corollary}
\label{c:image-ca-strong-irr}
Let $G$ be a group and let $A$ and $B$ be two finite sets.
Let $\tau \colon A^G \to B^G$ be a cellular automaton.
Then $\tau(A^G) \subset B^G$ is a strongly irreducible subshift.
\qed
\end{corollary}

From Proposition \ref{p:tau-strongly-is-strongly}, we also deduce that  strong irreducibility is a conjugacy invariant property:

\begin{corollary}
\label{c:strongly-inv}
Let $G$ be a group and let $A$ and $B$ be two finite sets.
Let $X \subset A^G$ and $Y \subset B^G$ be two conjugate subshifts.
Then $X$ is strongly irreducible if and only if $Y$ is strongly irreducible.
\qed
\end{corollary}

\section{Entropy of strongly irreducible subshifts}
\label{s:entropic-properties}

The following result is our main tool for studying entropic properties of strongly irreducible subshifts: 

\begin{lemma}
\label{l:ent-Z-smaller-ent-X}
Let $G$ be an amenable group, $A$ a finite set,  
and $\FF = (F_j)_{j \in J}$ a F\o lner net for $G$.
Let $X \subset A^G$ be a strongly irreducible subshift 
and let $\Delta$ be a finite subset of $G$ such that $1_G \in \Delta$ and $X$ is $\Delta$-irreducible.
Let $D$, $E$ and $E'$ be finite subsets of $G$ with $D^{+ \Delta} \subset E$.
Suppose that $T \subset G$ is an $(E,E')$-tiling and 
 that $Z$ is a subset of $X$ such that
 \begin{equation}
\label{e:pi-gD-Y-strict-in-pi-gD-X-lemma}
 Z_{gD} \subsetneqq  X_{gD} \quad \text{  for all } g \in T.
\end{equation}
 Then one has $\ent_\FF(Z) < \ent_\FF(X)$.
\end{lemma}

\begin{proof}
Consider, for each $j \in J$, the subset 
$T_j \subset T$ consisting of all  $ g \in T$ such that $gE \subset F_j$ (cf. Lemma \ref{l;tiling}).
Note that, for all $j \in J$ and $g \in T_j$, we have the inclusions $gD \subset gD^{+\Delta} \subset gE \subset F_j$, since $1_G \in \Delta$ and $D^{+ \Delta} \subset E$.
Given $j \in J$ and a subset $N \subset F_j$, let us denote by $\pi_{N}^{F_j} \colon A^{F_j} \to A^{N}$ the natural projection map.
Consider now,  for each $j \in J$, the subset 
$Q_j \subset  X_{F_j}$ defined by
$$
Q_j = \{q \in X_{F_j} : \pi_{gD}^{F_j}(q) \in Z_{gD} \text{ for all } g \in T_j\}.  
$$
Let us set $\rho = \vert X_{E} \vert$.  
Observe that  
\begin{equation}
\label{e:card-X-gE}
\vert X_{gE} \vert = \rho \quad \text{ for all } g \in G,
\end{equation}
since $X$ is $G$-invariant.
\par
We claim that
\begin{equation}
\label{xi-T-j}
  \vert Q_j \vert \leq (1 - \rho^{-1})^{|T_j|} \vert X_{F_j} \vert \quad \text{ for all } j \in J.
\end{equation}
To prove our claim, let us fix an element $j \in J$ and suppose that  $T_j = \{g_1, g_2, \ldots, g_m\}$, where $m = \vert T_j \vert$. 
Consider, for each $i \in \{0,1,\ldots,m\}$, 
the subset $Q_j^{(i)}  \subset X_{F_j} $
defined by
$$
Q_j^{(i)} = \{q \in X_{F_j}: \pi_{g_kD}^{F_j}(q) \in Z_{g_kD}  \text{ for all } 1 \leq k \leq i\}.
$$ 
  Note that  $Q_j^{(i)} \subset Q_j^{(i - 1)}$
 for all $i=1,2,\ldots, m$.
Let us show  that
\begin{equation}
\label{xi-T-j-i}
  \vert Q_j^{(i)} \vert \leq (1 - \rho^{-1})^{i} \vert X_{F_j} \vert
 \end{equation}
for all $i \in \{0,1,\ldots, m\}$.
This will prove \eqref{xi-T-j} since $Q_j^{(m)} = Q_j$.
\par
To establish \eqref{xi-T-j-i}, we proceed by induction on $i$.
For $i = 0$, we have $Q_j^{(i)} = X_{F_j}$ so that there is nothing to prove.
 Suppose now that $  \vert Q_j^{(i - 1)} \vert \leq (1 - \rho^{-1})^{i - 1} \vert X_{F_j} \vert$ for some $i \leq m - 1$.
 Consider the projection  $P_j^{(i - 1)} = \pi_{F_j \setminus g_iE}^{F_j}(Q_j^{(i - 1)})$ of $Q_j^{(i - 1)}$ on $A^{F_j \setminus g_iE}$.
As $Q_j^{(i - 1)} \subset P_j^{(i - 1)} \times X_{g_iE}$, we have
$\vert Q_j^{(i - 1)} \vert \leq  \vert P_j^{(i - 1)} \vert  \cdot \vert X_{g_iE} \vert$ and therefore
\begin{equation}
\label{e:lower-Pji}
\vert P_j^{(i - 1)} \vert \geq \rho^{-1} \vert Q_j^{(i - 1)} \vert,
\end{equation}
by using \eqref{e:card-X-gE}.
On the other hand, it follows from   \eqref{e:pi-gD-Y-strict-in-pi-gD-X-lemma} that  we can find 
a configuration $x_1  \in X$ such that $x_1\vert_{g_iD} \notin Z_{g_iE}$.
 As $(g_{i}D)^{+\Delta} = g_{i}D^{+\Delta} \subset g_i E$ and $X$ is $\Delta$-irreducible,
 we can find, for each $p \in P_j^{(i - 1)}$, a configuration
 $x \in X$ such that $x\vert_{F_j \setminus g_iE} = p$ and $x\vert_{g_iD} = x_1\vert_{g_iD}$.
 This shows that
 $$
 \vert Q_j^{(i -1)} \setminus Q_j^{(i)} \vert \geq \vert P_j^{(i - 1)} \vert.
 $$
Combining this inequality with  \eqref{e:lower-Pji}, we get
 $$
\vert  Q_j^{(i)} \vert \leq  \vert Q_j^{(i -1)} \vert - \vert P_j^{(i - 1)} \vert \leq 
(1 - \rho^{-1})\vert Q_j^{(i -1)} \vert, 
 $$
 which implies $\vert  Q_j^{(i)} \vert \leq (1- \rho^{-1})^i\vert X_{F_j}  \vert$ by our induction hypothesis. 
  This completes the proof of  \eqref{xi-T-j-i} and therefore of \eqref{xi-T-j}. 
\par
As $Z_{F_j} \subset Q_j$, we deduce from \eqref{xi-T-j} that 
$$
\vert Z_{F_j} \vert \leq (1 - \rho^{-1})^{|T_j|} \vert X_{F_j} \vert \quad \text{for all } j \in J.
$$
On the other hand, it follows from  Lemma \ref{l;tiling} that we can find a real number $\alpha > 0$ and an element $j_0 \in J$ such that
$\vert T_j \vert \geq \alpha \vert F_j\vert$ for all $j \geq j_0$.
This implies 
\begin{align*}
\frac{\log \vert Z_{F_j} \vert}{\vert F_j \vert} 
&\leq \frac{\log \vert X_{F_j} \vert}{\vert F_j \vert} + \frac{\vert T_j \vert}{\vert F_j \vert}\log(1 - \rho^{-1})\\
&\leq \frac{\log \vert X_{F_j} \vert}{\vert F_j \vert} + \alpha \log(1 - \rho^{-1})   
\end{align*}
for all $j \geq j_0$. Finally, by taking the limsup, this gives us 
$\ent_\FF(Z) \leq \ent_\FF(X) + \alpha \log(1 - \rho^{-1}) < \ent_\FF(X)$.
 \end{proof}

Let us give some direct applications of Lemma \ref{l:ent-Z-smaller-ent-X}.

\begin{proposition}
\label{p:entropy-increasing}
Let $G$ be an amenable group, $A$ a finite set,  
and $\FF = (F_j)_{j \in J}$ a F\o lner net for $G$.
Let $X \subset A^G$ be a strongly irreducible subshift.
Suppose that $Y \subset A^G$ is a subshift which is strictly contained in $X$.
Then one has $\ent_\FF(Y) < \ent_\FF(X) $. 
 \end{proposition}

\begin{proof}
As $Y \subsetneqq X$ and $Y$ is closed in $A^G$, 
we can find a finite subset $D \subset G$ such that $Y_D \subsetneqq X_D$.
By  the $G$-invariance of $X$ and $Y$,  this implies
 $Y_{gD} \subsetneqq  X_{gD}$
 for all $g \in G$.
\par
Let $\Delta$ be a finite subset of $G$ such that $1_G \in \Delta$ and $X$ is $\Delta$-irreducible, and 
take  $E = D^{+\Delta}$. By virtue of Lemma \ref{l;tilings-exist}, we can find a finite subset $E'  \subset G$ and an $(E,E')$-tiling  $T \subset G$.
Then, by taking $Z = Y$,  all the hypotheses in Lemma \ref{l:ent-Z-smaller-ent-X} are satisfied so that 
we get $\ent_\FF(Y) < \ent_\FF(X)$.
\end{proof}

\begin{corollary}
\label{c:injective-ca-surjective}
Let $G$ be an amenable group and let $\FF = (F_j)_{j \in J}$ be a F\o lner net for $G$.   
Let $A$ and $B$ be two finite sets.
Suppose that $X \subset A^G$ and $Y \subset B^G$ are two subshifts with $Y$ is strongly irreducible and $\ent_\FF(X) = \ent_\FF(Y)$.  
 Then every injective cellular automaton $\tau \colon X \to Y$ is surjective. 
 \end{corollary}

\begin{proof}
If $\tau \colon X \to Y$ is an injective cellular automaton,
then Proposition \ref{p:entropy-increasing} implies  that $\tau(X) = Y$, since the subshift $\tau(X) \subset Y$ satisfies
$\ent_\FF(\tau(X)) = \ent_\FF(X) = \ent_\FF(Y)$ by Corollary \ref{c:entropy-invariant} and our hypotheses on $X$ and $Y$.
\end{proof}

Given a group $G$ and a finite set $A$, a subshift $X \subset A^G$ is called \emph{surjunctive} if every injective cellular automaton $\tau \colon X \to X$ is surjective.
By taking $A = B$ and $X = Y$ in Corollary \ref{c:injective-ca-surjective},
we get the following result (which is also an immediate consequence of 
Theorem \ref{t:main-theorem} since injectivity implies pre-injectivity):

\begin{corollary}
\label{c:strongly-surjunctive}
Let $G$ be an amenable group and let $A$ be a finite set.
Then every strongly irreducible subshift $X \subset A^G$ is surjunctive.
\qed
\end{corollary}

From Lemma \ref{l:ent-Z-smaller-ent-X}, we can also deduce that non-trivial strongly irreducible subshifts have positive entropy:

\begin{proposition}
\label{p:positive-entropy}
Let $G$ be an amenable group,  $A$  a finite set,  
and $\FF = (F_j)_{j \in J}$ a F\o lner net for $G$.
Let $X \subset A^G$ be a strongly irreducible subshift containing at least two distinct configurations.
Then one has $\ent_\FF(X) > 0$. 
 \end{proposition}

Note that the previous statement is a direct consequence of Proposition \ref{p:entropy-increasing} in the case when there exists a subshift $Y \subset A^G$ such that
$\varnothing \subsetneqq Y \subsetneqq X$ (i.e., when $X$ is not \emph{minimal}), since this implies $0\leq \ent_\FF(Y) < \ent_\FF(X)$. For the general case, we need the following result which will also be used in the proof of Theorem \ref{t:pre-inj-si-source}:

\begin{lemma}
\label{l:strongly-implies-stab}
Let $G$ be a group and let $A$ be a finite set.
Let $\Delta$ be a finite subset of $G$ and let $X \subset A^G$ be a $\Delta$-irreducible subshift.
Suppose that $(\Omega_i)_{i \in I}$ is a (possibly infinite) family of (possibly infinite) subsets of $G$ such that
\begin{equation}
\label{e:family-Delta-separated}
\Omega_i^{+ \Delta} \cap \left(\bigcup_{k \in I \setminus \{i\}} \Omega_k \right) = \varnothing \quad \text{ for all } i \in I.
\end{equation}
Then, given any family $(x_i)_{i \in I}$ of configurations in $X$, there exists a configuration $x \in X$
which coincides with $x_i$ on $\Omega_i$ for all $i \in I$.  
\end{lemma}

\begin{proof}
In the case when the index set $I$ and the subsets $\Omega_i$ are all finite, the statement immediately follows from the definition of $\Delta$-irreducibility by induction on the cardinality 
of $I$.
\par
Let us now treat the general case.
Denote by $\PP_f(G)$ the set of all finite subsets of $G$.
 For each $\Lambda \in \PP_f(G)$, consider the subset $X(\Lambda) \subset X$ consisting of all configurations in $X$ which coincide 
 with $x_i$ on $\Lambda \cap \Omega_i$ for all $i \in I$. 
 First observe that $X(\Lambda)$ is closed in $X$  for each $\Lambda \in \PP_f(G)$
 by the properties of the prodiscrete topology.
 On the other hand, 
 if we fix  $\Lambda \in \PP_f(G)$, then the  subsets $\Psi_i = \Lambda \cap \Omega_i$  are all contained in $\Lambda$ and satisfy
\begin{equation*}
 \Psi_i^{+ \Delta} \cap \left(\bigcup_{k \in I \setminus \{i\}} \Psi_k \right) = \varnothing \quad \text{ for all } i \in I
\end{equation*}
by \eqref{e:family-Delta-separated}. As $\Lambda$ is finite, it follows that $X(\Lambda) \not= \varnothing$ by the first step in the proof.
As
$$
X(\Lambda_1) \cap X(\Lambda_2) \cap  \dots \cap X(\Lambda_n) = X(\Lambda_1 \cup \Lambda_2 \cup  \dots \cup \Lambda_n),
$$
we deduce that  $X(\Lambda_1) \cap X(\Lambda_2)\cap  \dots \cap X(\Lambda_n) \not= \varnothing$ for all 
$\Lambda_1,\Lambda_2,\dots , \Lambda_n \in \PP_f(G)$.
Thus,   $(X(\Lambda))_{\Lambda \in \PP_f(G)}$ is a family of closed subsets of $X$  with the finite intersection property.
By compactness of $X$, the intersection of this family is not empty.
This means that there exists a configuration $x \in X$ such that $x \in X(\Lambda)$ for each finite subset $\Lambda \subset G$.
Clearly, such an  $x$ has the required properties.
\end{proof}

 \begin{proof}[Proof of Proposition \ref{p:positive-entropy}]
Choose two  distinct configurations $x_0,x_1 \in X$.
Then there exists a finite subset $D \subset G$ such that $x_0\vert_D \not= x_1\vert_D$.
Note that this implies $(gx_0)\vert_{gD} \not= (gx_1)\vert_{gD}$ for all $g \in G$.
Let $\Delta$ be a finite subset of $G$ such that $X$ is $\Delta$-irreducible and $1_G \in \Delta$.
Let $E = D^{+\Delta}$.
By Lemma \ref{l;tilings-exist}, we can find a finite subset $E' \subset G$ and a $(E,E')$-tiling 
$T \subset G$.
Consider now the subset $Z \subset X$ consisting of all the configurations $z \in X$ such that, for all $g \in T$, 
one has either $z\vert_{gD} = (gx_0)\vert_{gD}$ or $z\vert_{gD} = (gx_1)\vert_{gD}$.
\par
By applying Lemma \ref{l:strongly-implies-stab} to the family $(gD)_{g \in T}$, we deduce that, given any map
 $\iota \colon T \to \{0,1\}$, there exists a configuration $x \in X$ such that
$x\vert_{gD} = (gx_{\iota (g)})\vert_{gD}$ for all $g \in T$.
 We deduce that
\begin{equation}
\label{e:card-Z-greater-2-power}
\vert Z_{F_j} \vert \geq 2^{\vert T_j \vert} \quad \text{ for all } j \in J,
\end{equation}
where, as above, $T_j = \{g \in T : gE \subset F_j \}$.
On the other hand, it follows from Lemma \ref{l;tiling} that there exist $\alpha > 0$ and $j_0 \in J$ such that 
$\vert T_j \vert \geq \alpha \vert F_j \vert$ for all $j \geq j_0$.
Using \eqref{e:card-Z-greater-2-power}, this gives us $\ent_\FF(Z) \geq \alpha \log 2$. 
As $Z \subset X$, this implies
$0 < \ent_\FF(Z) \leq \ent_\FF(X)$.
   \end{proof}

Combining  Proposition \ref{p:positive-entropy} and Corollary \ref{c:image-ca-strong-irr}, we get:   

   \begin{corollary}
   \label{c:image-ca-positive-ent}
   Let $G$ be an amenable group,  $A$ and $B$ two finite sets,   
   $\FF = (F_j)_{j \in J}$ a F\o lner net for $G$,
and $\tau \colon A^G \to B^G$ a non-constant cellular automaton.
Then one has $\ent_\FF(\tau(A^G)) > 0$.
\qed
   \end{corollary}

\begin{remark}
If $\tau \colon A^G \to B^G$ is a non-trivial cellular automaton as in Corollary 
\ref{c:image-ca-positive-ent}, 
then the subshift $\tau(A^G) \subset B^G$ is not minimal. Indeed, if $x_0 \in A^G$ is a constant configuration, then the subshift $Y = \{\tau(x_0)\}$ satisfies
$\varnothing \subsetneqq Y \subsetneqq \tau(A^G)$.
\end{remark}

\section{Proof of the main result}
\label{s:proof-main-result}

 Theorem \ref{t:main-theorem} will be deduced from the following statement:  

\begin{theorem}
\label{t:pre-inj-si-source}
Let $G$ be an amenable group 
and let $\FF = (F_j)_{j \in J}$ be a F\o lner net for $G$.
Let $A$ and $B$ be two finite sets.
Suppose that $X \subset A^G$   and $Y \subset B^G$ are subshifts
with $X$ strongly irreducible. 
Let $\tau \colon X \to Y$ be a pre-injective cellular automaton.
Then one has $\ent_\FF(\tau(X))  = \ent_\FF(X)$.
 \end{theorem}

\begin{proof}
We can assume $Y=\tau(X)$. We then have $\ent_\FF(Y)  \leq \ent_\FF(X)$ by Proposition \ref{p:ca-decrease-entropy}. Thus it suffices to show that $\ent_\FF(Y)  \geq \ent_\FF(X)$.
Suppose on the contrary that
\begin{equation}
\label{e:ent-Y-less-ent-X} 
\ent_\FF(Y) < \ent_\FF(X).
\end{equation}
 Let $\Delta$ be a finite subset of $G$ such that $X$ is $\Delta$-irreducible.
 After enlarging $\Delta$ if necessary,
 we can assume that $1_G \in \Delta$ and $\Delta = \Delta^{-1}$.
 We can also assume that $\Delta$ is a memory set for $\tau$. 
 Note that we have the inclusions $\Omega\subset\Omega^{+\Delta}\subset\Omega^{+\Delta^2}$ for every subset $\Omega\subset G$ since $1_G\in G$.
As  $Y_{F_j^{+\Delta^2}} \subset Y_{F_j} \times B^{F_j^{+ \Delta^2}\setminus F_j}$,
we have  
$$
\log \vert Y_{F_j^{+\Delta^2}} \vert \leq \log\vert Y_{F_j} \vert + 
\vert F_j^{+ \Delta^2} \setminus F_j \vert \cdot \log \vert B \vert
$$
for all $j \in J$,  and therefore
\begin{equation}
\label{e:limsup-Fj-Delta2}
\limsup_j \frac{\log \vert Y_{F_j^{+\Delta^2}} \vert }{\vert F_j \vert} 
\leq \limsup_j \frac{\log \vert Y_{F_j} \vert }{\vert F_j \vert} = \ent_\FF(Y),
 \end{equation}
since $\lim_j \vert F_j^{+ \Delta^2} \setminus F_j \vert/\vert F_j \vert = 0$ by \eqref{e:folner-net}. 
\par
From \eqref{e:limsup-Fj-Delta2} and \eqref{e:ent-Y-less-ent-X}, we deduce that there exists 
$j_0 \in J$ such that
\begin{equation}
\label{e:Y-Fj0-Delta2-less-X-j0}
\vert Y_{F_{j_0}^{+ \Delta^2}} \vert < \vert X_{F_{j_0}} \vert.
\end{equation}
 
Fix an arbitrary configuration $x_0 \in X$ and consider the finite subset $Z \subset X$ consisting of all configurations  $z \in X$ which coincide with $x_0$ outside of  $ F_{j_0}^{+\Delta}$.
We claim that  
\begin{equation}
\label{e:Z-f-X-F}
 X_{F_{j_0}} = Z_{F_{j_0}}.
\end{equation}
Indeed, let $x$ be an arbitrary configuration in $X$.   
As $X$ is $\Delta$-irreducible, it follows from Lemma \ref{l:strongly-implies-stab}, applied
by taking $I = \{1,2\}$, $\Omega_1 = F_{j_0}$ and $\Omega_2 = G \setminus  F_{j_0}^{+\Delta} $, 
that there exists a configuration $z \in X$ which coincides with $x$ on $F_{j_0}$ and with $x_0$ on $G \setminus  F_{j_0}^{+\Delta}$.
We then have $z \in Z$ and $x\vert_{F_{j_0}} = z\vert_{F_{j_0}}$.
This shows $X_{F_{j_0}} \subset Z_{F_{j_0}}$.
As $Z \subset X$, we also have $Z_{F_{j_0}} \subset X_{F_{j_0}}$
 and   \eqref{e:Z-f-X-F} follows.
 \par
 As the natural projection map $Z \to Z_{F_{j_0}}$ is surjective,
 we deduce from \eqref{e:Z-f-X-F} that
 \begin{equation}
 \label{e:card-Z-Fj0-Delta2-less}
 \vert X_{F_{j_0}} \vert \leq \vert Z \vert.
 \end{equation}
 Consider now an arbitrary configuration $z \in Z$. As $z$ and $x_0$ coincide outside of 
 $F_{j_0}^{+ \Delta}$ and $\Delta$ is a memory set for $\tau$,
 we know that $\tau(z)$ and $\tau(x_0)$ must coincide outside of $(F_{j_0}^{+ \Delta})^{+\Delta} = F_{j_0}^{+\Delta^2}$. Since $\tau(Z) \subset \tau(X) = Y$, this implies
\begin{equation}
\label{e:card-tau-Z-less}
\vert \tau(Z) \vert \leq \vert Y_{F_{j_0}^{+ \Delta^2}} \vert.
\end{equation}
From inequalities \eqref{e:Y-Fj0-Delta2-less-X-j0}, \eqref{e:card-Z-Fj0-Delta2-less}, and 
\eqref{e:card-tau-Z-less},
we deduce that $\vert \tau(Z) \vert < \vert Z \vert$.
This implies that  there exists two distinct configurations $z_1, z_2 \in Z$ such that 
$\tau(z_1) = \tau(z_2)$. As all configurations in $Z$ coincide outside of the finite set $F_{j_0}^{+\Delta}$, this shows that $\tau$ is not pre-injective. 
 \end{proof}

\begin{corollary}
\label{c:pre-inj-sis-implies surj}
Let $G$ be an amenable group 
and let $\FF = (F_j)_{j \in J}$ be a F\o lner net for $G$.
Let $A$ and $B$ be two finite sets.
Suppose that $X \subset A^G$ and $Y \subset B^G$ are strongly irreducible subshifts with
$\ent_\FF(X) = \ent_\FF(Y)$.
Then every pre-injective cellular automaton $\tau \colon X \to Y$ is surjective.
\end{corollary}

\begin{proof}
If $\tau \colon X \to Y$ is a pre-injective cellular automaton, then the subshift $\tau(X) \subset B^G$ must satisfy
$\ent_\FF(\tau(X)) = \ent_\FF(X)$ by Theorem \ref{t:pre-inj-si-source}.  
As $\tau(X) \subset Y$ and $\ent_\FF(X) = \ent_\FF(Y)$, this implies
$\tau(X) = Y$ by Proposition \ref{p:entropy-increasing}.
\end{proof}

\begin{proof}[Proof of Theorem \ref{t:main-theorem}]
It suffices to apply Corollary \ref{c:pre-inj-sis-implies surj} by taking $X = Y$.
\end{proof}


\section{Strongly irreducible subshifts over $\Z$}
\label{sec:strongly-over-Z}

In this section we present the proof of Proposition \ref{p;equivalences-strong-irreducibility} which gives a characterization of strongly irreducible subshifts over $\Z$ in terms of their associated languages.
\par
We first recall some basic definitions.
Let $A$ be a finite set.
We denote by  $A^*$ the free monoid based on $A$.
Thus, $A^*$ is the set   consisting
  of all finite words  with leters in  $A$ equipped with the multiplicative law given by  the concatenation product of words 
  (the identity element in $A^*$ is the empty word).
  \par 
Consider  now a subshift $X \subset A^\Z$.
 The \emph{language} of $X$ is the subset $L(X) \subset A^*$  consisting of all words $w \in A^*$
 which can be written in the form $w = x(1)x(2) \cdots x(n)$
for some $x \in X$ and   $n \geq 0$.
It is well known that a subshift $X \subset A^\Z$ is
  topologically mixing if and only if the following conditions is satisfied
  \begin{enumerate}[{\rm (TM)}]
  \item{for all $u,v \in L(X)$ there exists
an  integer $  n_0(u,v) \geq 0$
 such that, for every integer  $n \geq n_0(u,v)$, there exists a word $w \in A^*$ of length $n$ satisfying $uwv \in L(X)$}
 \end{enumerate} (see for instance \cite[Example 6.3.3, Definition 4.5.9 and Exercise 6.3.5)]{lind-marcus}).
\par
We are now in position to prove the characterization of strongly irreducible subshifts
given in Proposition \ref{p;equivalences-strong-irreducibility}.

\begin{proof}[Proof of Proposition \ref{p;equivalences-strong-irreducibility}]
Suppose (a). 
Let $\Delta$ be a finite subset of $\Z$ such that $X$ is $\Delta$-irreducible. 
Choose an integer  $N_0 \geq 0$ such that $\Delta \subset [-N_0,N_0]$. Let $u =a_1a_2 \cdots a_n,v = b_1b_2 \cdots b_m \in L(X)$ and $N \geq N_0$. 
Then the sets $\Omega_1 = \{-n-N+1, -n-N+2, \ldots, -N\}$ and $\Omega_2 = \{1,2,\ldots,m\}$  
satisfy
 $\Omega_1^{+\Delta} \cap \Omega_2 = \varnothing$. Since $u,v \in L(X)$, we can find configurations $x_1, x_2 \in X$ such that $x_1(-N-n+i) = a_i$ and $x_2(j) = b_j$ for all $i\in \{1,2,\ldots,n\}$ and $j \in \{1,2,\ldots,m\}$. 
By $\Delta$-irreducibility of $X$, there exists 
$x \in X$ such that $x\vert_{\Omega_1} =  x_1\vert_{\Omega_1}$ and $x\vert_{\Omega_2} =  x_2\vert_{\Omega_2}$. 
Then the word $w =   x(-N+1)x(- N + 2) \cdots x(0)$ has length $N$ and satisfies 
 $uwv \in L(X)$. This shows that (a) implies (b).
\par
Conversely, suppose (b).  Let us show that $X$ is $\Delta$-irreducible for $\Delta = \{-N_0, -N_0+1, \ldots, N_0\}$.
Let $\Omega_1$ and $\Omega_2$ be two finite subsets of $\Z$ such that 
\begin{equation}
\label{e;intersection-vide}
\Omega_1^{+\Delta} \cap \Omega_2 = \varnothing
\end{equation} 
and let $x_1, x_2 \in X$. We want to show that there exists a
configuration $x \in X$ such that 
\begin{equation}
\label{e;restrictions}
x\vert_{\Omega_1} = x_1\vert_{\Omega_1} \mbox{ and } x\vert_{\Omega_2} = x_2\vert_{\Omega_2}.
\end{equation} 

First observe that from \eqref{e;intersection-vide} we deduce that
\begin{equation*}
\label{e;intersection-simple-vide}
\Omega_1 \cap \Omega_2 = \varnothing
\end{equation*}
since $0 \in \Delta$. 
Moreover, the condition $\Omega_1^{+\Delta} \cap \Omega_2 = \varnothing$ implies $\Omega_2^{+\Delta} \cap \Omega_1 = \varnothing$, 
since $\Delta = -\Delta$. 
Thus, after possibly exchanging $\Omega_1$ and $\Omega_2$,
we can assume $\min \Omega_1 < \min \Omega_2$.
On the other hand, after enlarging $\Omega_2$ if necessary,  we can also assume 
$\max \Omega_1 < \max \Omega_2$.
\par 
Now observe that condition (b) implies, by an immediate induction on $s$,  the following:
\begin{enumerate}
\item[{\rm (b')}] 
there exists an  integer $N_0 \geq 0$ such that, for any sequence of words $u_1,u_2, \ldots, u_s \in L(X)$, $s \geq 1$, and any sequence of integers $N_1,N_2, \ldots, N_{s-1} \geq N_0$, there exist words 
$w_1, w_2, \ldots , w_{s - 1} \in A^*$, with $w_i$ of length $N_i$ for $1 \leq i \leq s - 1$, satisfying 
\begin{equation}
\label{e;uwvzu}
u_1w_1 u_2w_2\cdots u_{s-1}w_{s-1} u_s \in L(X).
\end{equation}
\end{enumerate}
 \par
 Let us introduce the following equivalence relation
$\sim_1$ on $\Omega_1$ (resp. $\sim_2$ on $\Omega_2$). Given $\omega_1, \omega_1' \in \Omega_1$ (resp. $\omega_2, \omega_2' \in \Omega_2$) we write $\omega_1 \sim_1 \omega_1'$ (resp. $\omega_2 \sim_2\omega_2'$) if there is no element of $\Omega_2$
(resp. of $ \Omega_1$) between $\omega_1$ and $\omega_1'$ (resp. between $\omega_2$ and $\omega_2'$). 
Note that the conditions $\min \Omega_1 < \min \Omega_2$ and 
$\max \Omega_1 < \max \Omega_2$ imply that $\sim_1$ and $\sim_2$ have the same number of equivalence classes, say $s$.
Let
$$
\Omega_1 = \bigcup_{i=1}^s \Omega_{1,i} \text{\  and \ } \Omega_2 = \bigcup_{i=1}^s \Omega_{2,i}
$$
be the corresponding partitions of $\Omega_1$ and $\Omega_2$ into equivalence classes.
Let us set $m_i = \min \Omega_{1,i}$, $n_i = \max \Omega_{1,i}$, 
$p_i = \min \Omega_{2,i}$ and $q_i = \max \Omega_{2,i}$ for $i=1,2,\ldots,s$.  
We then have, after renumbering the equivalence classes if necessary,  
 $$
m_1 \leq n_1 < p_1 \leq q_1 < m_2 \leq n_2 < p_2 \leq q_2  < \dots <
m_s \leq n_s < p_s \leq q_s.
$$
It follows  from \eqref{e;intersection-vide} that we have $N_i = p_i-n_i > N_0$ for all $i = 1,2,\dots,s$ and   $M_i = m_i-q_{i-1} > N_0$ for all $i=2,3,\ldots, s$.
\par
Consider the words $u_i  ,v_i \in L(X)$ defined by
$$
u_i = x_1(m_i)x_1(m_i + 1) \cdots x_1(n_i)
\quad \text{ and } \quad
v_i = x_2(p_i)x_2(p_i + 1) \cdots x_2(q_i)
$$
for $i = 1,2,\dots,s$.
\par By applying (b') to the sequence of words $u_1,v_1,u_2,v_2, \ldots, u_s,v_s \in L(X)$ and to the sequence
of integers $N_1,M_1,N_2, M_2, \ldots,  N_{s-1},  M_{s-1}, N_s$
we deduce that we can find
words $w_1, z_1,w_2,z_2, \ldots, w_{s-1},z_{s-1},w_s \in A^*$ with
$w_i$ of length $N_i$, for $i=1,2, \ldots, s$, and   $z_i$ of length $M_i$, for $i=1,2,\ldots, s-1$,
such that the word
$$
w = u_1w_1v_1z_1 u_2w_2v_2z_2 \cdots u_{s-1}w_{s-1}v_{s-1}z_{s-1} u_sw_sv_s
$$
belongs to $L(X)$.
Writing $w = a_1a_2 \cdots a_\ell$, where $\ell = q_s - m_1+1$ and $a_1,a_2, \ldots, a_\ell \in A$, this implies that we can find $x \in X$ satisfying 
$x(m_1 + k - 1) = a_k$ for $k = 1,2,\dots,\ell$.
Then $x$ satisfies \eqref{e;restrictions}.
This shows that (b) implies  (a).
\end{proof}

Let $A$ be a finite set. 
Given a finite $A$-labeled graph $\GG$, the set $X_\GG \subset A^\Z$, consisting of all configurations in $A^\Z$ which can be represented by some bi-infinite path in $\GG$, is a subshift of $A^\Z$.  A subshift $X \subset A^\Z$ is said to be \emph{sofic} if there exists a finite $A$-labeled graph $\GG$ such that $X = X_\GG$ (see, e.g., \cite[Chapter 3]{lind-marcus}). We are now in position to prove Corollary \ref{c:sofic-strongly-irr}:

\begin{proof}[Proof of  Corollary \ref{c:sofic-strongly-irr}]
The necessity follows from Proposition \ref{p:strongly-irred-implies-top-mixing}.
\par
Conversely, let $\GG$ be a finite  $A$-labeled graph such that $X = X_\GG$ and denote by $Q$ its vertex set. As every topologically mixing subshift over $\Z$ is irreducible, we may suppose that $\GG$ is strongly connected, that is, for all $q,q' \in Q$ there exists a path $\pi$ in $\GG$ which connects $q$ to $q'$ (see, e.g. \cite[Lemma 3.3.10]{lind-marcus}).
\par
It can be shown (see, e.g., \cite[Proposition 3.3.2, Proposition 3.3.9 and Proposition 3.3.16]{lind-marcus}) that $\GG$ can be chosen such that, in addition, there exists a \emph{synchronized word} for $\GG$, that is, a word $u_0 \in L(X)$ for which there exists a vertex $q_0 = q_0(u_0) \in Q$ such that all paths representing $u_0$ terminate at $q_0$.
\par
  By the topological mixing property, we can find an integer $n_0 = n_0(u_0,u_0)$ such that for every 
$n \geq n_0$, there exists a word $w \in A^*$ of length $n$ satisfying $u_0wu_0 \in L(X)$. Now, every path $\pi$ in $\GG$ representing the word $u_0wu_0$ factorizes as 
$\pi = \pi_1 \pi' \pi_2$, where the paths $\pi_1,\pi_2$ (resp. $\pi'$) represent $u_0$
(resp. $w$) and terminate (resp. starts) at $q_0$, 
so that  the path $\varphi = \pi' \pi_2$ is a closed path based at the vertex $q_0$. 
Thus, setting $L_0 = n_0 + \ell_0$, where $\ell_0$ is the length of $u_0$, we deduce that, 
 for every $N \geq L_0$, there exists a closed path $\varphi$ in $\GG$ of length $N$ based at $q_0$.
\par
Let $D$ denote the diameter of $\GG$, that is, the length of the longest geodesic path  in $\GG$, and set $N_0 = L_0 + 2D$.
It is then clear that, for all $q, q' \in Q$ and every $N \geq N_0$ there exists a path
$\psi$ in $\GG$ of length $N$ starting at $q$, passing through $q_0$, and terminating at $q'$.
Let us show that $N_0$ satisfies condition (b) in 
Proposition \ref{p;equivalences-strong-irreducibility}.
Let $u,v \in L(X)$ and $N \geq N_0$. Choose paths $\pi_u, \pi_v$ in $\GG$ representing
$u$ and $v$ respectively. Let $q$ (resp. $q'$) denote the terminal vertex of $\pi_u$ (resp. the starting vertex of $\pi_v$).
Then we can find a path $\psi$ in $\GG$ of length $N$ connecting $q$ to $q'$.
It follows that the word $w$ represented by $\psi$ has length $N$ and satisfies $uwv \in L(X)$. 
By applying  Proposition \ref{p;equivalences-strong-irreducibility}, we deduce  that the subshift $X$ is strongly irreducible.
\end{proof}

As every subshift of finite type $X \subset A^\Z$ is sofic 
(see, e.g., \cite[Theorem 3.1.5]{lind-marcus}), we deduce the following:

\begin{corollary}
\label{c:dernier}
Let $A$ be a finite set and let $X \subset A^\Z$ be a subshift of finite type. 
Then $X$ is strongly irreducible if and only if it is topologically mixing.
\qed
\end{corollary}

\begin{remark}
In \cite[Proposition 3.39(2)]{kurka} it is shown that a Markov subshift (that is,
a subshift of finite type over $\Z$ defined by a set of forbidden words of length at most
two) is topologically mixing if and only if it satisfies condition (b) in Proposition \ref{p;equivalences-strong-irreducibility}. This result is covered by Corollary
\ref{c:dernier} since from this corollary we deduce that condition (b) in Proposition \ref{p;equivalences-strong-irreducibility} and condition (TM) above are
equivalent for subshifts of finite type over $\Z$.
\end{remark}

\bibliographystyle{siam}
\bibliography{myhill}

\end{document}